\documentclass[11pt,reqno]{amsart}
\usepackage{amssymb}
\usepackage{amsfonts}
\usepackage{color,graphics,epsfig,amsmath,a4wide,psfrag}
\usepackage{amsmath}
\usepackage{graphicx}

\providecommand{\U}[1]{\protect\rule{.1in}{.1in}}

\numberwithin{equation}{section}

\newcommand{\calA}{\mathcal{A}}

\newcommand{\mC}{\mathbb{C}}

\newtheorem{theorem}{Theorem}[section]
\newtheorem{lemma}[theorem]{Lemma}
\newtheorem{corollary}[theorem]{Corollary}
\newtheorem{proposition}[theorem]{Proposition}

\theoremstyle{definition}

\theoremstyle{definition}
\newtheorem{definition}[theorem]{Definition}
\theoremstyle{definition}

\begin{document}
\title[Algebraic-analytic properties of $A_{0}+AP_{+}$]{A subalgebra of the Hardy algebra relevant in control theory and its
algebraic-analytic properties}
\author{Marie Frentz and Amol Sasane}
\address{Department of Mathematics, London School of Economics, Houghton Street, London
WC2A 2AE, United Kingdom}
\email{e.m.frentz@lse.ac.uk, sasane@lse.ac.uk}

\begin{abstract}
We denote by $A_{0}+AP_{+}$ the Banach algebra of all complex-valued functions
$f$ defined in the closed right halfplane, such that $f$ is the sum of a
holomorphic function vanishing at infinity and a ``causal'' almost periodic
function. We give a complete description of the maximum ideal space
${\mathfrak{M}}(A_{0}+AP_{+})$ of $A_{0}+AP_{+}$. Using this description, we
also establish the following results:

\begin{enumerate}
\item The corona theorem for $A_{0}+AP_{+}$.

\item ${\mathfrak{M}}(A_{0}+AP_{+})$ is contractible (which  implies that
$A_{0}+AP_{+}$ is a projective free ring).

\item $A_{0}+AP_{+}$ is not a GCD domain.

\item $A_{0}+AP_{+}$ is not a pre-Bezout domain.

\item $A_{0}+AP_{+}$ is not a coherent ring.

\end{enumerate}
The study of the above algebraic-anlaytic properties is motivated by
applications in the frequency domain approach to linear control theory, where they play an important role in the
stabilization problem.

\end{abstract}
\maketitle

\section{Introduction}

The aim in this paper is to study the algebraic-analytic properties of a new
Banach algebra which is relevant in control theory as a class of stable
transfer functions.

This Banach algebra, denoted by $A_{0}+AP_{+}$, which is defined below, is a
bigger Banach algebra than the well known algebra $\mathcal{A}_{+}$
 used widely in control theory since the 1970s (see \cite{CalDes}, \cite{CZ}),
for systems described by PDEs and delay-differential equations. Recall that
$\mathcal{A}_{+}$ consists of all Laplace transforms of complex Borel measures
$\mu$ on $\mathbb{R}$ with support contained in the half-line $[0,+\infty)$,
and such that $\mu$ does not have a singular non-atomic part. Let
$\mathbb{C}_{\scriptscriptstyle\geq0}:=\{s\in\mathbb{C}:\text{Re}(s)\geq0\}$.
Then by the Lebesgue decomposition, it follows that
\[
\mathcal{A}_{+}=\left\{  \widehat{\mu}:\mathbb{C}_{\scriptscriptstyle\geq
0}\rightarrow\mathbb{C}:%
\begin{array}
[c]{ll}%
\widehat{\mu}(s)=\widehat{f_{a}}(s)+\displaystyle\sum_{ k\geq
0}f_{k}e^{-st_{k}}\;(s\in\mathbb{C}_{\scriptscriptstyle\geq0}), & \\
f_{a}\in L^{1}[0,+\infty),\;(f_{k})_{k\geq0}\in\ell^{1}, & \\
t_{0}=0<t_{1},t_{2},t_{3},\cdots. &
\end{array}
\right\}  .
\]
$\mathcal{A}_{+}$ is a Banach algebra with pointwise operations, and the norm taken as the total variation of
the measure, namely
\[
\Vert\widehat{\mu}\Vert_{\mathcal{A}_{+}}:=|\mu|([0,+\infty))=\Vert f_{a}%
\Vert_{L^{1}[0,+\infty)}+\Vert(f_{k})_{k\geq0}\Vert_{\ell^{1}},\quad
\widehat{\mu}\in\mathcal{A}_{+}.
\]
The algebra $\calA_+$ is relevant in control theory as a
class of stable transfer functions because it maps $L^{p}$ inputs to $L^{p}$
outputs for all $1\leq p\leq+\infty$: indeed, if $\widehat{\mu}\in
\mathcal{A}_{+}$, $1\leq p\leq+\infty$, $u\in L^{p}[0,+\infty)$, then
$y:=\mu\ast u\in L^{p}[0,+\infty)$ and moreover,
\[
\sup_{0\neq u\in L^{p}}\frac{\Vert y\Vert_{p}}{\Vert u\Vert_{p}}\leq
\Vert\widehat{\mu}\Vert_{\mathcal{A}_{+}}.
\]
In fact, if $p=1$ or $p=+\infty$, then we have equality above.

On the other hand, another widely used Banach algebra in control theory
serving as a class of stable transfer functions is the Hardy algebra
$H^{\infty}$ of the half-plane, consisting of all bounded and holomorphic
functions $f$ defined in the open right half-plane
$$
\mathbb{C}_{\scriptscriptstyle>0}:=\{s\in\mathbb{C}:\text{Re}(s)>0\},
$$
again with
pointwise operations, but now with the supremum norm:
\[
\Vert f\Vert_{\infty}=\sup_{s\in\mathbb{C}_{\scriptscriptstyle>0}}|f(s)|,\quad
f\in H^{\infty}.
\]
It is well-known that if $f\in H^{\infty}$ and $u\in L^{2}[0,+\infty)$, then
$y$ defined by $\widehat{y}(s):=f(s)\widehat{u}(s)$ ($s\in\mathbb{C}%
_{\scriptscriptstyle>0}$) is such that $y\in L^{2}[0,+\infty)$. Moreover, we
have
\[
\sup_{0\neq u\in L^{2}}\frac{\Vert y\Vert_{2}}{\Vert u\Vert_{2}}= \Vert
f\Vert_{\infty}.
\]
Clearly $\mathcal{A}_{+}\subset H^{\infty}$, but from the point of view of
control theory, $H^{\infty}$ is an unnecessarily large object to serve as a
class of stable transfer functions, since it includes functions that can
hardly be considered to represent any physical system.

In this article we consider the closure of $\mathcal{A}_{+}$ in $H^{\infty}$
in the $\|\cdot\|_{\infty}$, and this will be our Banach algebra $A_{0}%
+AP_{+}$. We give the precise definition below.

\begin{definition}
\label{snow}Let
\[
A_{0}:=\left\{  f:\mathbb{C}_{\scriptscriptstyle\geq0}\rightarrow\mathbb{C}:\;%
\begin{array}
[c]{l}%
f\text{ is holomorphic in }\mathbb{C}_{\scriptscriptstyle>0}\\
f\text{ is continuous on }\mathbb{C}_{\scriptscriptstyle\geq0}\\
\displaystyle \lim_{s\rightarrow\infty}f(s)=0
\end{array}
\right\}  ,
\]
and let the set of causal almost periodic functions $AP_{+}$ be the closure of
$\textrm{span}\{e^{-\;\! \cdot\;\! t},t\geq0\}$ in the $L^{\infty}$-norm. In this paper we
consider the following class of functions
\[
A_{0}+AP_{+}=\{f:\mathbb{C}_{\scriptscriptstyle\geq0}\rightarrow\mathbb{C\ }:f=f_{\scriptscriptstyle A_{0}%
}+f_{\scriptscriptstyle AP_{+}},\ f_{\scriptscriptstyle A_0}\in A_{0},\ f_{\scriptscriptstyle AP_+}\in AP_{+}\},
\]
with pointwise operations and with the norm $\|f_{\scriptscriptstyle A_{0}}+f_{\scriptscriptstyle AP_{+}%
}\|_{\infty}.$
\end{definition}

More precisely, we consider classes of unstable control systems with transfer
functions belonging to the field of fractions $\mathbb{F}(A_{0}+AP_{+}%
)$. Our first main result is the following
corona theorem for $A_{0}+AP_{+}$:

\begin{theorem}
\label{coronathm}
Let $n,d\in A_{0}+AP_{+}$. The following are equivalent;
\begin{enumerate}
\item There exist $x,y\in A_{0}+AP_{+}$ such that $nx+dy=1$.

\item There exist $\delta>0$ such that $|n(s)|+|d(s)|\geq\delta>0$ for all
$s\in\mathbb{C}_{\scriptscriptstyle \geq0}$.
\end{enumerate}
\end{theorem}

A similar theorem can also be shown for matricial data, but for the sake of simplicity,
we just prove the result for a pair of functions. To prove this theorem,
we characterize the maximal ideal space $ \mathfrak{M}(A_{0}+AP_{+})$ of the Banach algebra
$A_{0}+AP_{+}$; see Theorem~\ref{maxideal}. The
next main result is the following:

\begin{theorem}
\label{Contractible}$\mathfrak{M}(A_{0}+AP_{+})$ is contractible.
\end{theorem}

With the corona theorem for $A_{0}+AP_{+}$, and the contractability of
$\mathfrak{M}(A_{0}+AP_{+})$, we prove a number of results, all of them concerning 
algebraic properties of $A_{0}+AP_{+}$:
\begin{enumerate}
\item $AP_{+}+A_{0}$ is a Hermite ring.
\item $AP_{+}+A_{0}$ is projective free.
\item $AP_{+}+A_{0}$ is not a GCD domain.
\item $AP_{+}+A_{0}$ is not a pre-Bezout domain.
\item $AP_{+}+A_{0}$ is not a coherent ring.
\end{enumerate}
For the definitions of the above algebraic properties,  we refer the reader to Section~\ref{prel}.
From the point of view of  applications, we now briefly mention the relevance of these results in the solution of the stabilization
problem in control theory:
\begin{enumerate}
 \item Hermiteness: since
the ring $A_{0}+AP_{+}$ is Hermite, if a transfer function $G$ has a left (or
right) coprime factorization, then $G$ has a doubly coprime factorization, and
the standard Youla parametrization yields all stabilizing controllers for $G,$
see \cite[Theorem 66, p. 347]{V}.
\item Projective freeness: since
the ring $A_{0}+AP_{+}$ is projective free, a plant is stabilizable if and only if
it admits a right (or left) coprime factorization, see \cite{Q}.
\item Pre-Bezout property: Every
transfer function $p\in\mathbb{F}(R)$ admits a coprime factorization if and
only if $R$ is a pre-Bezout domain, see \cite[Corollary 4]{Q2}.
\item GCD domain: Every
transfer function $p\in\mathbb{F}(R)$ admits a weak coprime factorization if
and only if $R$ is a GCD domain, see \cite[Corollary 3]{Q2}.
\item Coherence: For implications
of (the lack of) the coherence property we refer to
\cite{Q}.
\end{enumerate}
The paper is organized as follows:
\begin{enumerate}
 \item In Section 2, we introduce the relevant
notation  used throughout in the article.
\item In Section 3, we characterize the
maximal ideal space of $A_{0}+AP_{+}$.
\item In Section 4, we use the characterization of $\mathfrak{M}(A_{0}+AP_{+})$ to prove a
corona theorem for $A_{0}+AP_{+}$.
\item In Section 5, we prove that $A_{0}+AP_{+} $
is contractible and as corollaries we get that $A_{0}+AP_{+}$ is Hermite and
projective free.
\item In Section 6, we prove that $A_{0}+AP_{+}$ is not a GCD
domain.
\item In Section 7, that $A_{0}+AP_{+}$ is not a pre-Bezout domain.
\item Finally, in Section 8, we prove that $A_{0}+AP_{+}$ is not coherent.
\end{enumerate}

\section{Preliminaries and notation\label{prel}}

Throughout the article, we consider the Banach algebra
$A_{0}+AP_{+}$, defined in Definition \ref{snow}, with pointwise operations
and supremum norm. We point out that elements $f_{\scriptscriptstyle AP_{+}}$ in $AP_{+}$ can be
written%
\[
f_{\scriptscriptstyle AP_{+}}(s)=\sum_{k\geq 0}f_{k}e^{-st_{k}},
\]
where $t_{0}=0<t_{1},t_{2},\cdots $.

 To prove the corona theorem for $A_{0}+AP_{+}$, we
shall use Kronecker's approximation theorem, see for instance \cite[Chapter
23]{HW}.

\begin{theorem}
\label{kronecker}
Let $\{\tau_{1},\cdots ,\tau_{n}\}$ be a set of rationally
independent real numbers. Then, given $\varepsilon>0$ and $(\delta_{1}%
,\cdots ,\delta_{n})\in\mathbb{R}^{n},$ there exist $(m_{1},\cdots,m_{n}%
)\in\mathbb{Z}^{n}$ and $\xi\in\mathbb{R}$ such that%
\[
|\delta_{i}-\xi\tau_{i}-m_{i}|<\varepsilon\text{ \ \ \ for all }%
i\in\{1,\cdots,n\}.
\]
\end{theorem}

We continue with a few definitions which explain the results presented in the introduction.

\begin{definition}
\label{contract}A topological space $X$ is said to be \emph{contractible} if
there exist a continuous map $H:X\times\lbrack0,1]\rightarrow X$ and an
$x_{0}\in X$ such that for all $x\in X,\ H(x,0)=x$ and $H(x,1)=x_{0}.$
\end{definition}

\begin{definition}
Let $R$ be a ring with an identity element. A matrix $f\in R^{n\times k}$ is
called {\em left invertible} if there exists a $g\in R^{k\times n}$ such that
$gf=I_{k}.$ The ring $R$ is called a \emph{Hermite ring} if for all
$k,n\in\mathbb{N}$, $k<n,$ and for all left invertible matrices $f\in
R^{n\times k}$, there exist $F,G\in R^{n\times n}$ such that $GF=I_{n}$ and
$F_{ij}=f_{ij}$ for all $i\in\{1,\cdots,n\},\ j\in\{1,\cdots,k\}$.
\end{definition}

\begin{definition}
Let $R$ be a commutative ring with identity. Then $R$ is
\emph{projective free} if every finitely generated projective $R$-module is
free. Recall that an $R$-module $M$ is called
\begin{enumerate}
\item  { \em  free} if $M$ $\cong$ $R^{d}$ for some integer
$d\geq0$;
\item  {\em projective} if there exists an $R$-module $N$ and
an integer $d\geq0$ such that $M\oplus N$ $\cong R^{d}$.
\end{enumerate}
\end{definition}

It can be shown that every projective-free ring is Hermite, see for example 
\cite{BS}. 

\begin{definition}
Let $R$ be an integral domain, that is, a commutative unital ring having no
divisors of zero. Then
\begin{enumerate}
\item An element $d\in R$ is called a {\em greatest common divisor} (gcd)
of $a,b\in R$ if it is a divisor of $a$ and $b$, and, moreover, if $k$ is
another divisor, then $k$ divides $d$.

\item The integral domain $R$ is said to be a \emph{GCD domain} if
for all $a,b\in R,$ there exists a greatest common divisor $d$ of $a,b$.

\item  The integral domain $R$ is said to be a \emph{pre-Bezout
domain} if for every $a,b\in R$ for which there exists a gcd $d$, there exist
$x,y\in R$ such that $d=xa+yb.$

\item  The ring $R$ is called \emph{coherent} if for any pair $(I,J)$
of finitely generated ideals in $R$, their intersection $I\cap J$ is finitely
generated again.
\end{enumerate}
\end{definition}

We also recall the definition of a bounded approximate identity, which
will play an important role when we prove that $A_{0}+AP_{+}$ is not a GCD
domain, not a pre-Bezout domain and not coherent.

\begin{definition}
\label{AI}Let $R$ be a commutative Banach algebra without identity element.
Then $R$ has a \emph{bounded left approximate identity} if there exists a
bounded sequence $(e_{n})_{n\geq1}$ of elements $e_{n}\in R$ such that for all
$f\in R$, $\lim_{n\rightarrow\infty}\|e_{n}f-f\|_{\mathcal{\infty}}=0$.
\end{definition}

\section{The maximal ideal space of $A_{0}+AP_{+}$}

In this section we give a characterization of  the maximal ideal
space of $A_{0}+AP_{+}$. Before we state this theorem, we need to introduce
some notation. For $s\in\mC_{\scriptscriptstyle \geq 0}$ let $\underline{s}$ denote
point evaluation at $s$, that is, $\underline{s}(f)=f(s)$, $f\in A_{0}+AP_{+}$.

\begin{definition}
$\chi:\mathbb{R}\rightarrow\mathbb{C}$ is called a {\em character} if
 $
|\chi(t)|=1$ and $\chi(t+\tau)=\chi(t)\chi(\tau)$ for
all $t,\tau\in\mathbb{R}$.
\end{definition}

\begin{theorem}
\label{maxideal}If $\varphi$ is in the maximal ideal space of $A_{0}+AP_{+}$,
then one of the following three statements holds:

\begin{itemize}
\item[(a)] There exists an $s\in\mathbb{C}_{\scriptscriptstyle \geq0}$ such that $\varphi(f)=f(s)
$ for all $f\in A_{0}+AP_{+}$.

\item[(b)] We have $\varphi=\underline{+\infty}$, that is,
$$
\varphi(f)= \underline{+\infty}(f):=\displaystyle \lim_{s\rightarrow
\infty}f(s)=f_{0},
$$
for all $f=f_{\scriptscriptstyle A_{0}}+\displaystyle \sum_{k\geq0}f_{k}e^{-\;\! \cdot\;\! t_{k}}\in
A_{0}+AP_{+}$,  where $f_{\scriptscriptstyle A_{0}}\in A_{0}$.

\item[(c)] There exist a $\sigma\geq0$ and a character $\chi$ such that
$$
\varphi(f)=\displaystyle \sum_{k\geq0}%
f_{k}e^{-\sigma t_{k}}\chi(t_{k}),
$$ for
all $f=\displaystyle \sum_{k\geq0}f_{k}e^{-\;\! \cdot\;\! t_{k}}+f_{\scriptscriptstyle A_{0}}\in A_{0}+AP_{+}, $ with
$f_{\scriptscriptstyle A_{0}}\in A_{0}$.
\end{itemize}
\end{theorem}
\begin{proof}
We note that via the conformal map
\[
s\rightarrow\frac{s-1}{s+1},
\]
we can map $\mathbb{D}$ onto $\mathbb{C}_{\scriptscriptstyle \geq0}.$ In particular, this means
that
\[
f\in A(\mathbb{D})\Longleftrightarrow f\left(  \frac{s-1}{s+1}\right)  \in
A_{0}+\mathbb{C}.
\]
Since the maximal ideal space of $A(\mathbb{D})$ is
$\overline{\mathbb{D}}$ (point evaluations on the closed unit
disc; see for instance \cite[p. 283]{Rud91}), it follows that the maximal ideal space of
$A_{0}+\mathbb{C}$ is $\underline{\mathbb{C}_{\scriptscriptstyle \geq0}}\cup\{\underline{+\infty
}\}.$

Suppose that $\varphi\in\mathfrak{M}(A_{0}+AP_{+})\ \backslash
\ (\underline{\mathbb{C}_{\scriptscriptstyle \geq0}}\cup\{\underline{+\infty}\}).$ First, we
shall show that if $f_{\scriptscriptstyle A_0}\in A_{0}$, then $\varphi(f_{\scriptscriptstyle A_0})=0$. Assume on the
contrary that $\varphi(f_{\scriptscriptstyle A_0})\neq 0$.  Then $\varphi|_{A_{0}+\mathbb{C}}%
\in\mathfrak{M}(A_{0}+C)=\underline{\mathbb{C}_{\scriptscriptstyle \geq0}}\cup
\{\underline{+\infty}\}$. Note that if $\varphi|_{A_{0}+\mathbb{C}}%
=\underline{+\infty}$, then
$$
\varphi(f_{\scriptscriptstyle A_0})=\displaystyle \lim_{s\rightarrow+\infty}
f_{\scriptscriptstyle A_0}(s)=0,
$$
which is a contradiction. So $\varphi|_{A_{0}+\mathbb{C}}\neq \underline{+\infty}$,
and there exists $s_{0}\in\mathbb{C}_{\scriptscriptstyle \geq0}$
such that $\varphi|_{A_{0}+\mathbb{C}}=\underline{s_{0}}$. Moreover, we note
that $A_{0}$ is an ideal in $A_{0}+AP_{+}$. Thus for $F\in A_{0}+AP_{+}$, we have
\[
\varphi(F)=\frac{\varphi(F\cdot f_{0})}{\varphi(f_{0})}=\frac{(F\cdot
f_{0})(s_{0})}{f_{0}(s_{0})}=F(s_{0})=\underline{s_{0}}(F).
\]
But this is a contradiction since $\varphi\in\mathfrak{M}%
(A_{0}+AP_{+})\ \backslash\ (\underline{\mathbb{C}_{\scriptscriptstyle \geq0}}\cup
\{\underline{+\infty}\})$.

Hence $\varphi|_{AP_{+}}$ is a nontrivial
complex homomorphism. From the known characterization of the maximal ideal space of $AP_+$
(see for example \cite[Theorem 4.1]{AS56}), we obtain that there exist a $\sigma\geq0,$ and a character  $\chi$ such that
$\varphi|_{AP_{+}}=\varphi_{\sigma,\chi}$, that is,
\[
\varphi_{\sigma,\chi}\left(  \sum_{k\geq0}f_{k}e^{-\;\! \cdot\;\!  t_{k}}\right)
=\sum_{k\geq0}f_{k}e^{-  \sigma t_{k}}\chi(t_{k}).
\]
 This completes the proof.
\end{proof}

\section{Corona theorem for $A_{0}+AP_{+}$}

We will now give a proof of the corona theorem, Theorem~\ref{coronathm}. The
proof relies on the characterization of the maximal ideal space of $A_{0}%
+AP_{+}$ provided in Theorem~\ref{maxideal}.

\begin{proof}[Proof of Theorem~\ref{coronathm}] Suppose that the corona condition holds, that is, that
there exist $\delta>0$ such that $|n(s)|+|d(s)|\geq\delta>0$ for all
$s\in\mathbb{C}_{\geq0}.$ Assume that (1) in Theorem~\ref{coronathm}
does not hold. Then the ideal $\langle n, d \rangle$ generated by $n,d$ is
not the whole ring, and so there is a maximal ideal which contains it.  Thus
there exists $\varphi\in\mathfrak{M}(A_{0}+AP_{+})$ such that
\[
\left\langle n,d\right\rangle \subset\ker(\varphi)\subset A_{0}+AP_{+}\text{.}%
\]
In particular, $\varphi(n)=\varphi(d)=0$. By Theorem
\ref{maxideal} we know that the maximal ideal space of $A_{0}+AP_{+}$ consists
of three different types of homomorphisms. However, $\varphi$ cannot be of
type (a), since then there exists $s\in\mathbb{C}_{\scriptscriptstyle \geq0}$ such that%
\[
\varphi(n)=n(s)=0 \textrm{ and } \varphi(d)= d(s)=0,
\]
which contradicts the corona condition. Moreover, $\varphi$ cannot be of type
(b), since then 
\[
\varphi(n)=\lim_{s\rightarrow+\infty} n(s)=0 \textrm{ and } \varphi(d)=\lim_{s\rightarrow+\infty}d(s) =0,
\]
which again violates the corona condition. Therefore, $\varphi$ must be of
type (c). Let 
\begin{align*}
n &  =n_{\scriptscriptstyle A_{0}}+n_{\scriptscriptstyle AP_{+}}=n_{\scriptscriptstyle A_{0}}(s)+\displaystyle\sum_{k\geq0}n_{k}e^{-st_{k}%
}\ \ \ \text{and}\\
d &  =d_{\scriptscriptstyle A_{0}}+d_{\scriptscriptstyle AP_{+}}=d_{A_{0}}(s)+\displaystyle \sum_{k\geq0}d_{k}e^{-st_{k}}.
\end{align*}
There exist $\sigma\geq0$ and a character $\chi$ such that,
\begin{align*}
0 &  =\varphi(n)=\sum_{k\geq0}n_{k}e_{{}}^{-\sigma t_{k}}\chi(t_{k}),\text{
\ \ and}\\
0 &  =\varphi(d)=\sum_{k\geq0}d_{k}e^{-\sigma\tau_{k}}\chi(\tau_{k}).
\end{align*}
Since $n$, $d\in A_{0}+AP_{+}$, for every choice of $\delta,M>0,$ there exists
$N\in\mathbb{N}$ such that%
\begin{align*}
\left\Vert \sum_{k=0}^{N}n_{k}e^{-st_{k}}-n_{AP_{+}}\right\Vert _{\infty}  &
<\frac{\delta}{M},\\
\left\Vert \sum_{k=0}^{N}d_{k}e^{-s\tau_{k}}-d_{AP_{+}}\right\Vert _{\infty}
& <\frac{\delta}{M}.
\end{align*}
Therefore
\begin{align*}
\left\vert \varphi\left(  \sum_{k=0}^{N}n_{k}e^{-st_{k}}-n_{AP_{+}}\right)
\right\vert  & \leq\left\Vert \varphi\right\Vert \cdot\left\Vert \sum
_{k=0}^{N}n_{k}e^{-st_{k}}-n_{AP_{+}}\right\Vert _{\infty}\leq\frac{\delta}%
{M},\\
\left\vert \varphi\left(  \sum_{k=0}^{N}d_{k}e^{-s\tau_{k}}-d_{AP_{+}}\right)
\right\vert  & \leq\left\Vert \varphi\right\Vert \cdot\left\Vert \sum
_{k=0}^{N}d_{k}e^{-s\tau_{k}}-d_{AP_{+}}\right\Vert _{\infty}\leq\frac{\delta
}{M}.
\end{align*}
In particular, this means that
\[
\left\vert \sum_{k=0}^{N}n_{k}e_{{}}^{-\sigma t_{k}}\chi(t_{k})\right\vert
<\frac{\delta}{M}\text{ \ \ and \ \ }\left\vert \sum_{k=0}^{N}d_{k}e_{{}%
}^{-\sigma\tau_{k}}\chi(\tau_{k})\right\vert <\frac{\delta}{M}.
\]
We shall now show that there exists a $\omega_{\ast}\in\mathbb{R}$ such that
\[
\left\vert \sum_{k=0}^{N}n_{k}e_{{}}^{-\sigma t_{k}}e^{-i\omega_{\ast}t_{k}%
}\right\vert <\frac{\delta}{2M}\text{ \ \ and \ \ }\left\vert \sum_{k=0}%
^{N}d_{k}e_{{}}^{-\sigma\tau_{k}}e^{-i\omega_{\ast}\tau_{k}}\right\vert
<\frac{\delta}{2M}.
\]
Choose $\xi_{1},\cdots,\xi_{K}$ rationally independent real numbers such that
 the real numbers $t_{1},\cdots,t_{N}$, $\tau_1,\cdots, \tau_N$ can be written as
\[
t_{j}=\sum_{k=1}^{K}\alpha_{jk}\xi_{k} \textrm{ and }\tau_{j}=\sum_{k=1}^{K}\widetilde{\alpha}_{jk}\xi_{k} , \quad j=1, \dots N,
\]
for appropriate integers $\alpha_{jk}, \widetilde{\alpha}_{jk}$. Since $|\chi(t)|=1$ for all $t\in\mathbb{R},$ we
may set $\chi(\xi_{k})=e^{2\pi i\delta_{k}}$ for some $\delta_{k}\in
\mathbb{R}$. Then,
\[
\chi(t_{k})=\exp\left(  2\pi i\sum_{k=1}^{K}\alpha_{jk}\delta_{k}\right) \textrm{ and }
\chi(\tau_{k})=\exp\left(  2\pi i\sum_{k=1}^{K}\widetilde{\alpha}_{jk}\delta_{k}\right)  .
\]
By Theorem \ref{kronecker}, for all $\eta>0$ there exist a $\beta\in
\mathbb{R}$ and real numbers $m_{1},\cdots ,m_{N}$ such that%
\[
|\beta\xi_{i}-\delta_{i}-m_{i}|<\eta,\text{ \ \ for }i=1,\cdots ,N.
\]
Hence
\begin{align*}
\left\vert \chi(t_{k})-\exp(2\pi i\beta t_{k})\right\vert  &  =\left\vert
\exp\left(  2\pi i\sum_{k=1}^{K}\alpha_{jk}\delta_{k}\right)  -\exp\left(  2\pi
i\beta\sum_{k=1}^{K}\alpha_{jk}\xi_{k}\right)  \right\vert \\
&  \leq2\pi\sum_{k=1}^{K}|\alpha_{jk}|\eta.
\end{align*}
Similarly, $\displaystyle
\left\vert \chi(\tau_{k})-\exp(2\pi i\beta \tau_{k})\right\vert \leq2\pi\sum_{k=1}^{K}|\widetilde{\alpha}_{jk}|\eta$. Let
 $
C=\displaystyle \sup_{j\in\{1,\ldots ,N\}}2\pi\sum_{k=1}^{K}\max\{|\alpha_{jk}|,|\widetilde{\alpha}_{jk}|\}$.
Then for $\eta$
small enough,
\[
\left\vert \sum_{k=0}^{N} \left(a_{k}\chi(t_{k})-a_{k}e^{2\pi i\beta t_{k}%
}\right) \right\vert \leq C\eta\sum_{k=0}^{N}|a_{k}|<\frac{\delta}{2M}.
\]
Using this we obtain
\begin{align*}
\left\vert n_{\scriptscriptstyle AP_{+}}(\sigma+i\omega_{\ast})\right\vert  & \leq\left\vert
n_{\scriptscriptstyle AP_{+}}(\sigma+i\omega_{\ast})-\sum_{k=0}^{N}n_{k}e_{{}}^{-(\sigma
+i\omega_{\ast})t_{k}}\right\vert +\left\vert \sum_{k=0}^{N}n_{k}e_{{}%
}^{-(\sigma+i\omega_{\ast})t_{k}}\right\vert \\
& \leq\frac{\delta}{M}+\frac{\delta}{2M}=\frac{3\delta}{2M}.
\end{align*}
and also, $\left\vert d_{\scriptscriptstyle AP_{+}}(\sigma+i\omega_{\ast})\right\vert \leq \frac{3\delta}{2M}$.
But the choice of $M$ was arbitrary, and so, the above shows that
 for every $\delta>0,$ there exists $\omega_{\ast}\in
\mathbb{R}$ such that%
\begin{equation}
\left\vert \sum_{k=0}^{\infty}n_{k}e_{{}}^{-\sigma t_{k}}e^{-i\omega_{\ast
}t_{k}}\right\vert <\frac{\delta}{8}\text{ \ \ \ and \ \ }\left\vert
\sum_{k=0}^{\infty}d_{k}e_{{}}^{-\sigma\tau_{k}}e^{-i\omega_{\ast}\tau_{k}%
}\right\vert <\frac{\delta}{8}.\label{p1}%
\end{equation}
Hence $|n_{\scriptscriptstyle AP_{+}}(\sigma+i\omega_{\ast})|+|d_{\scriptscriptstyle AP_{+}}%
(\sigma+i\omega_{\ast})|<\delta/4.$ Now, we note that both $n_{\scriptscriptstyle AP_{+}}%
(\sigma+i\;\! \cdot\;\! )$ and $d_{\scriptscriptstyle AP_{+}}(\sigma+i\;\! \cdot\;\! )$ are almost periodic
since%
\[
\omega\longmapsto n_{\scriptscriptstyle AP_{+}}(\sigma+i\omega)=\sum_{k=0}^{\infty}%
\widetilde{n}_{k}e^{-i\omega t_{k}}%
\]
and%
\[
\omega\longmapsto d_{\scriptscriptstyle AP_{+}}(\sigma+i\omega)=\sum_{k=0}^{\infty}%
\widetilde{d}_{k}e^{-i\omega t_{k}},
\]
for $\widetilde{n}_{k}=n_{k}e^{-\sigma t_{k}},\widetilde{d}_{k}=d_{k}%
e^{-\sigma\tau_{k}}$. But we recall the classical result that for any almost periodic function $F$,
for every $\epsilon>0$
there exists $L>0$ such that every interval of length $L$ has a $\lambda
=\lambda(L)$ such that for all $\omega\in\mathbb{R}$%
\[
|F(\omega)-F(\omega+\lambda)|<\epsilon.
\]
We can use this fact to construct a sequence $(\omega_{n})_{n=1}^{\infty}$ such
that $\omega_{n}\rightarrow+\infty$ as $n\rightarrow\infty$ and%
\begin{align}
\left\vert n_{\scriptscriptstyle AP_{+}}(\sigma+i\omega_{\ast})-n_{\scriptscriptstyle AP_{+}}(\sigma+i\omega
_{n})\right\vert  &  <\frac{\delta}{8},\nonumber\\
\left\vert d_{\scriptscriptstyle AP_{+}}(\sigma+i\omega_{\ast})-d_{\scriptscriptstyle AP_{+}}(\sigma+i\omega
_{n})\right\vert  &  <\frac{\delta}{8}.\label{p2}%
\end{align}
Moreover, by definition of $A_{0},$
\begin{equation}
\lim_{n\rightarrow+\infty}n_{\scriptscriptstyle A_{0}}(\sigma+i\omega_{n})=\lim_{n\rightarrow
+\infty}d_{\scriptscriptstyle A_{0}}(\sigma+i\omega_{n})=0.\label{p3}%
\end{equation}
Thus, taking $s=\sigma+i\omega_{n}$ in the corona condition we get%
\[
|n(s)|+|d(s)|\leq\underset{I}{\underbrace{|n_{\scriptscriptstyle A_{0}}(s)|+|d_{\scriptscriptstyle A_{0}}(s)|}%
}+\underset{II}{\underbrace{|n_{\scriptscriptstyle AP_+}(s)|+|d_{\scriptscriptstyle AP_+}(s)|}}.
\]
But using (\ref{p3}) we can make $I$ as small as we please, and using
(\ref{p1})-(\ref{p2}), $II$ is smaller than $\delta/2.$ That is,%
\[
|n(s)|+|d(s)|\leq\delta,
\]
which is a contradiction. Thus, (1) in Theorem~\ref{coronathm} must hold whenever (2) in
Theorem~\ref{coronathm} (the corona condition) holds. That (1) implies (2) in Theorem~\ref{coronathm}
is obvious and this concludes the proof.
\end{proof}

\section{Contractability of $\mathfrak{M}(A_{0}+AP_{+})$}

We will show that the maximal ideal space $\mathfrak{M}%
(A_{0}+AP_{+})$ of $A_{0}+AP_{+}$  is contractible (Theorem~\ref{Contractible}). Then, as
corollaries, we obtain that $A_{0}+AP_{+}$ is Hermite and projective free.

To prove that $\mathfrak{M}(A_{0}+AP_{+})$ is contractible, we will proceed in
several steps, and we start with a few lemmas which mainly concern the
topology of $\mathfrak{M}(A_{0}+AP_{+})$.

\begin{lemma}
\label{close}The set $\mathfrak{M}(A_{0}+AP_{+})\backslash
\underline{\mathbb{C}_{\scriptscriptstyle \geq 0}}$ is closed in $\mathfrak{M}(A_{0}+AP_{+})$.
\end{lemma}
\begin{proof}
It is enough to show that $\underline{\mathbb{C}_{\scriptscriptstyle \geq0}}$ is open. Let
$\widehat{\cdot}$ denote the Gelfand transform. As a consequence of Theorem
\ref{maxideal}, $\varphi\in\underline{\mathbb{C}_{\scriptscriptstyle \geq0}}$ if and only if
there exists $f_{\scriptscriptstyle A_{0}}\in A_{0}$ such that%
\[
\left\vert \widehat{f}_{\scriptscriptstyle A_{0}}(\varphi)\right\vert =\left\vert \varphi\left(
\widehat{f}_{\scriptscriptstyle A_{0}}\right)  \right\vert >0.
\]
Therefore,
\[
\underline{\mathbb{C}_{\scriptscriptstyle \geq0}}=\bigcup\limits_{f_{\scriptscriptstyle A_{0}}\in A_{0}}\left\{
\varphi\in\mathfrak{M}(A_{0}+AP_{+}):\left\vert \widehat{f}_{\scriptscriptstyle A_{0}}%
(\varphi)\right\vert >0\right\}  ,
\]
and since the union of open sets is open, $\underline{\mathbb{C}_{\scriptscriptstyle \geq0}}$ is
open, which completes the proof.
\end{proof}

\begin{lemma}
\label{homeo}$\underline{\mathbb{C}_{\scriptscriptstyle \geq0}}$ is homeomorphic to
$\mathbb{C}_{\scriptscriptstyle \geq0}.$
\end{lemma}

\begin{proof}
The map $\iota:\underline{\mathbb{C}_{\scriptscriptstyle \geq0}}\rightarrow\mathbb{C}_{\scriptscriptstyle \geq0}$
given by $\underline{s}\longmapsto s$ is onto. Further it is injective since,
if $\underline{s_{1}}=\underline{s_{2}},$ then%
\[
\underline{s_{1}}\left(  \frac{1}{1+s}\right)  =\frac{1}{1+s_{1}}=\frac
{1}{1+s_{2}}=\underline{s_{2}}\left(  \frac{1}{1+s}\right)  .
\]
That is, $s_{1}=s_{2}.$ Therefore, $\iota$ is invertible. Let $(s_{\alpha})$
be a net such that $s_{\alpha}\rightarrow s_{0}.$ Since elements in
$A_{0}+AP_{+}$ are continuous in $\mathbb{C}_{\scriptscriptstyle \geq0}$ by definition, it
follows that, for $f\in A_{0}+AP_{+},$ $f(s_{\alpha})\rightarrow f(s_{0}).$
That is%
\[
\underline{s_{\alpha}}(f)\rightarrow\underline{s_{0}}(f).
\]
Since this holds for arbitrary $f$ it follows that $\underline{s_{\alpha}%
}\rightarrow\underline{s_{0}}$ in $\underline{\mathbb{C}_{\scriptscriptstyle \geq0}}$. What
remains to prove is that the inverse is continuous, but if
$\underline{s_{\alpha}}\rightarrow\underline{s_{0}}$, then%
\[
\underline{s_{\alpha}}\left(  \frac{1}{1+s}\right)  =\frac{1}{1+s_{\alpha}%
}\rightarrow\frac{1}{1+s_{0}}=\underline{s_{0}}\left(  \frac{1}{1+s}\right)  .
\]
Thus, $s_{\alpha}\rightarrow s_{0}$ in $\mathbb{C}_{\scriptscriptstyle \geq0}$ so the inverse is
continuous and $\underline{\mathbb{C}_{\scriptscriptstyle \geq0}}$ is indeed homeomorphic to
$\mathbb{C}_{\scriptscriptstyle \geq0}.$
\end{proof}

\begin{lemma}
\label{inflim}If $(\underline{s_{\alpha}})$ is a net in $\underline{\mathbb{C}%
_{\scriptscriptstyle \geq0}}$ which converges in $\mathfrak{M}(A_{0}+AP_{+})$ to a $\varphi
\in\mathfrak{M}(A_{0}+AP_{+})\backslash\underline{\mathbb{C}_{\scriptscriptstyle \geq0}}$ then
$s_{\alpha}\rightarrow\infty.$
\end{lemma}

\begin{proof}
Note that, in particular,%
\[
\underline{s_{\alpha}}\left(  \frac{1}{1+s}\right)  =\frac{1}{1+s_{\alpha}%
}\rightarrow\varphi\left(  \frac{1}{1+s}\right)  =0,
\]
since $\frac{1}{1+s}$ is not invertible. Therefore $s_{\alpha}\rightarrow
\infty$.
\end{proof}

\begin{lemma}
\label{lims}Let $\varphi\in\mathfrak{M}(A_{0}+AP_{+})\backslash
\underline{\mathbb{C}_{\scriptscriptstyle \geq0}}$ and let $(\varphi_{\alpha})$ be a net such that;
\begin{enumerate}
\item for all $f\in A_{0},\ \varphi_{\alpha}(f)\rightarrow0.$

\item for all $T>0,\ \varphi_{\alpha}\left(  e^{-sT}\right)  \rightarrow
\varphi\left(  e^{-sT}\right)  .$
\end{enumerate}
Then $\varphi_{\alpha}\rightarrow\varphi$ in $\mathfrak{M}(A_{0}%
+AP_{+})\backslash\underline{\mathbb{C}_{\scriptscriptstyle \geq0}}.$
\end{lemma}
\begin{proof}
By hypothesis, for every $f_{\scriptscriptstyle A_{0}}\in A_{0}$ and for every exponential
polynomial%
\[
P(s)=\sum_{k=0}^{N}f_{k}e^{-st_{k}},\ \ \ 0=t_{0}<t_{1},t_{2},\ldots,t_{N},
\]
we have that $\varphi_{\alpha}(f_{A_{0}}+P)\rightarrow\varphi(f_{A_{0}}+P)$
since $\varphi_{\alpha}$ and $\varphi$ are homomorphisms. Let%
\[
f=f_{\scriptscriptstyle A_{0}}+\sum_{k\geq0}f_{k}e^{-\;\! \cdot\;\! t_{k}}\in A_{0}+AP_{+}.
\]
Then, for every $\epsilon>0$, we can chose an exponential polynomial $P$
such that%
\begin{equation}
\|f-f_{\scriptscriptstyle A_{0}}-P\|_{\infty}=\left\Vert \sum_{k\geq 0}f_{k}e^{-\;\! \cdot
\;\! t_{k}}-P\right\Vert_{\infty}\leq\frac{\epsilon}{4}%
.\label{afton}%
\end{equation}
Moreover, since $\varphi_{\alpha}(f_{\scriptscriptstyle A_{0}}+P)\rightarrow\varphi(f_{\scriptscriptstyle A_{0}}+P)$
there exists $\alpha_{\ast}$ such that for all $\alpha>\alpha_{\ast}$ there
holds%
\begin{equation}
\left\vert \varphi_{\alpha}(f_{\scriptscriptstyle A_{0}}+P)-\varphi(f_{\scriptscriptstyle A_{0}}+P)\right\vert
<\frac{\epsilon}{2}.\label{kvall}%
\end{equation}
Combining (\ref{afton}) and (\ref{kvall})%
\begin{align*}
\left\vert \varphi_{\alpha}(f)-\varphi(f)\right\vert  &  =\left\vert
\varphi_{\alpha}(f_{\scriptscriptstyle A_{0}}+P+f-
f_{\scriptscriptstyle A_{0}}-P)-\varphi(f_{\scriptscriptstyle A_{0}}+P+f-f_{\scriptscriptstyle A_{0}%
}-P)\right\vert \\
&  \leq\left\vert \varphi_{\alpha}(f_{\scriptscriptstyle A_{0}}+P)-\varphi(f_{\scriptscriptstyle A_{0}%
}+P)\right\vert +\left\vert (\varphi_{\alpha}-\varphi)(f-f_{\scriptscriptstyle A_{0}%
}-P)\right\vert \\
&  \leq\frac{\varepsilon}{2}+\|\varphi_{\alpha}-\varphi\|\cdot\|f-f_{\scriptscriptstyle A_{0}%
}-P\|_{\infty}\\
&  \leq\frac{\varepsilon}{2}+2\frac{\varepsilon}{4}=\varepsilon.
\end{align*}
That is $\varphi_{\alpha}(f)\rightarrow\varphi(f)$ for all $f\in A_{0}+AP_{+}$
and consequently $(\varphi_{\alpha})$ converges in the weak $\ast$-topology on
$\mathfrak{M}(A_{0}+AP_{+}),$ which completes the proof.
\end{proof}

With these tools available, we shall now prove that $A_{0}+AP_{+}$ is contractible.

\begin{proof}
[Proof of Theorem 1.3]Recall that $\mathfrak{M}(A_{0}+AP_{+})$ consist of
three kinds of homomorphisms, as stated in Theorem \ref{maxideal}. We will
define a map $H:\mathfrak{M}(A_{0}+AP_{+})\times\lbrack0,1]\rightarrow
\mathfrak{M}(A_{0}+AP_{+})$ and show that this map satisfies the necessary
properties in Definition \ref{contract}. In particular, we define $H$ as follows;

\begin{enumerate}
\item For $\underline{s}\in\underline{\mathbb{C}_{\scriptscriptstyle \geq0}}$, define
 $H(\underline{s},t)=s-\log(1-t)$ for $t\in\lbrack0,1)$, and
 $H(\underline{s},1)=\underline{+\infty}$.

\item For $\underline{+\infty}$ we define $H(\underline{+\infty}%
,t)=\underline{+\infty}$ for all $t\in\lbrack0,1].$

\item For $\varphi\in\mathfrak{M}(A_{0}+AP_{+})\backslash
(\underline{\mathbb{C}_{\scriptscriptstyle \geq0}}\cup \underline{+\infty})$ there exists
$\sigma>0$ and a character $\chi$ such that $\varphi=\varphi_{\sigma,\chi}$,
where $\varphi_{\sigma,\chi}$ is defined as in (c) in Theorem \ref{maxideal}.
We define $H(\varphi_{\sigma,\chi},t)=\varphi_{\sigma-\log(1-t),\chi}$ for
 $t\in\lbrack0,1)$ and $H(\underline{s},1)=\underline{+\infty}$.
\end{enumerate}
From the definition of $H$ it is obvious that the choice of $x_{0}$ in
Definition \ref{contract} should be $\underline{+\infty}$. We shall now prove
that $H$ is continuous. Every net $(\varphi_{\alpha},t_{\alpha})$, with
elements in $\mathfrak{M}(A_{0}+AP_{+})\times\lbrack0,1]$ can be partitioned
onto three subnets:
\begin{itemize}
\item[$1^\circ$] one where $(\varphi_{\alpha},t_{\alpha})=(\underline{s_{\alpha
}},t_{\alpha})\in\underline{\mathbb{C}_{\scriptscriptstyle \geq0}}\times\lbrack0,1]$ for each
$\alpha,$
\item[$2^\circ$] one where $(\varphi_{\alpha},t_{\alpha})=(\underline{+\infty
},t_{\alpha})\in\{\underline{+\infty}\}\times\lbrack0,1]$ for each $\alpha,$ and
\item[$3^\circ$] one where $(\varphi_{\alpha},t_{\alpha})=(\varphi
_{\sigma_{\alpha},\chi_{\alpha}},t_{\alpha})\in\left(  \mathfrak{M}%
(A_{0}+AP_{+})\backslash(\underline{\mathbb{C}_{\scriptscriptstyle \geq0}}\cup\underline{+\infty
})\right)  \times\lbrack0,1]$ for each $\alpha.$
\end{itemize}
It is therefore enough to prove that, for these three kinds of nets,
if $(\varphi_{\alpha},t_{\alpha})$ converges to $(\varphi,t)$ in $\mathfrak{M}%
(A_{0}+AP_{+})\times\lbrack0,1],$ then $(H(\varphi_{a},t_{\alpha}))$ converges
to $H(\varphi,t)$ in $\mathfrak{M}(A_{0}+AP_{+})$. We shall treat these cases separately.

\medskip

\noindent {\bf Case when $(\varphi_{\alpha},t_{\alpha})=(\underline{s_{\alpha}},t_{\alpha})\in
\underline{\mathbb{C}_{\scriptscriptstyle \geq0}}\times\lbrack0,1]$.} Firstly, if
$t_{\alpha}\equiv1,$ then $t=1$ by necessity, and%
\[
H(\underline{s_{\alpha}},t_{\alpha})=H(\underline{s_{\alpha}},1)=\underline{+\infty}
=H(\varphi,1)=H(\varphi,t),
\]
so we are done. Therefore, assume that each $t_{\alpha}\in\lbrack0,1),$ and
thus%
\[
H(\underline{s_{\alpha}},t_{\alpha})=\underline{s_{\alpha}-\log(1-t_{\alpha}%
)}.
\]
Now consider the following three cases.
\begin{itemize}
\item[(1a)] $\varphi=\underline{s}$ for some $s\in\mathbb{C}_{\scriptscriptstyle \geq0}.$ If
$t\in\lbrack0,1),$ then%
\[
H(\underline{s},t)=\underline{s-\log(1-t)}.
\]
But, since $\underline{s_{\alpha}}\rightarrow\underline{s}$ we know by Lemma
\ref{homeo} that $s_{\alpha}\rightarrow s$ in $\mathbb{C}_{\scriptscriptstyle \geq0},$ and since
$\log(1-t)$ is continuous for $t\in\lbrack0,1),$ $-\log(1-t_{\alpha
})\rightarrow-\log(1-t)$ and%
\[
s_{\alpha}-\log(1-t_{\alpha})\rightarrow s-\log(1-t)\text{ \ \ in }%
\mathbb{C}_{\scriptscriptstyle \geq0}.
\]
Hence, using Lemma \ref{homeo} once more we see that $H(\underline{s_{\alpha}%
},t_{\alpha})\rightarrow H(\underline{s},t).$

On the other hand, if $t=1$, then $H(\underline{s},t)=H(\underline{s}%
,1)=\underline{+\infty},$ and since $t_{\alpha}\rightarrow1$ we have that
$\ \operatorname{Re}(s_{\alpha}-\log(1-t_{\alpha}))\rightarrow\infty.$ The
elements $f\in A_{0}+AP_{+}$ are given by%
\begin{equation}
f=f_{\scriptscriptstyle A_{0}}+\sum_{k\geq0}f_{k}e^{-\;\! \cdot\;\!  t_{k}},\label{fett}%
\end{equation}
so $f(s_{\alpha}-\log(1-t_{a}))\rightarrow f_{0}$ when $\operatorname{Re}%
(s_{\alpha}-\log(1-t_{\alpha}))\rightarrow\infty,$ which corresponds to
evaluation at infinity. Since the choice of $f\in$ $A_{0}+AP_{+}$ was
arbitrary, also in this case $H(\underline{s_{\alpha}},t_{\alpha})\rightarrow
H(\underline{s},t).$

\item[(1b)] $\varphi=\underline{+\infty}.$ In this case, $H(\varphi
,t)=\underline{+\infty}.$ Note that, by Lemma \ref{inflim}, if
$\underline{s_{\alpha}}\rightarrow\underline{+\infty},$ then $s_{\alpha
}\rightarrow\infty$ and $s_{\alpha}-\log(1-t_{\alpha})\rightarrow\infty.$ In
particular, for all $f_{\scriptscriptstyle A_{0}}\in A_{0}$%
\[
f_{\scriptscriptstyle A_{0}}(s_{\alpha}-\log(1-t_{\alpha}))\rightarrow0=\underline{+\infty
}(f_{\scriptscriptstyle A_{0}}),
\]
since elements in $A_{0}$ have limit zero at infinity. Moreover, for all $T>0,
$%
\begin{align*}
\underline{s_{\alpha}}(e^{-sT})=e^{-s_{\alpha}T}\rightarrow &
0=+\underline{\infty}(e^{-sT})\text{ \ \ and}\\
\underline{s_{\alpha}-\log(1-t_{\alpha})}(e^{-sT})  &  \rightarrow
0=+\underline{\infty}(e^{-sT}).
\end{align*}
Hence, using Lemma \ref{lims} we see that $H(\underline{s_{\alpha}},t_{\alpha
})\rightarrow H(\underline{+\infty},t).$

\item[(1c)] $\varphi=\varphi_{\sigma,\chi}.$ Since $\underline{s_{\alpha}%
}\rightarrow\varphi_{\sigma,\chi}\in\mathfrak{M}(A_{0}+AP_{+})\backslash
\underline{\mathbb{C}_{\scriptscriptstyle \geq0}}$ we have that $s_{\alpha}\rightarrow\infty$ by
Lemma \ref{inflim} and therefore $s_{\alpha}-\log(1-t_{\alpha})\rightarrow
\infty.$ Let $f\in A_{0}+AP_{+}$ with $f$ as in (\ref{fett}) be given (and
arbitrary). Then, arguing as above and for $t=1$, we have%
\[
H(\varphi,t)=H(\varphi_{\sigma,\chi},1)=\underline{+\infty}%
\]
and%
\[
H(s_{\alpha},t_{\alpha})\rightarrow f_{0}=\underline{+\infty}(f).
\]
Since this holds for all $f\in A_{0}+AP_{+}$ this implies that
$H(\underline{s_{\alpha}},t_{\alpha})\rightarrow H(\varphi_{\sigma,\chi},t).$

If $t<1,$ then for $T>0,$%
\[
H(\varphi_{\sigma,\chi},t)(e^{-sT})=e^{-(\sigma-\log(1-t))T}\chi(T).
\]
Since $\underline{s_{\alpha}}\rightarrow\varphi_{\sigma,\chi},$ we have that%
\[
\underline{s_{\alpha}}(e^{-sT})=e^{-s_{\alpha}T}\rightarrow e^{-\sigma T}%
\chi(T)=\varphi_{\sigma,\chi}(e^{-sT}).
\]
In particular, this means that%
\[
e^{-(s_{\alpha}-\log(1-t_{\alpha}))T}\rightarrow e^{-(\sigma-\log(1-t_{\alpha
}))T}\chi(T),
\]
which by the definition of $H$ implies that $H(\underline{s_{\alpha}},t_{\alpha
})\rightarrow H(\varphi_{\sigma,\chi},t).$
\end{itemize}
Now, the cases (1a)-(1c) prove that $H$ is continuous when $(\varphi_{\alpha
},t_{\alpha})=(\underline{s_{\alpha}},t_{\alpha}).$

\medskip

\noindent {\bf Case when $(\varphi_{\alpha},t_{\alpha})=(\underline{+\infty},t_{\alpha})\in
\{\underline{+\infty}\}\times\lbrack0,1]$.} This case is trivially
satisfied, since $\varphi=\underline{+\infty}$ in this case, and
$H(\varphi_{\alpha},t_{\alpha})=H(\underline{+\infty},t_{\alpha}%
)=\underline{+\infty}=H(\varphi,t)$.

\medskip

\noindent {\bf Case when $(\varphi_{\alpha},t_{\alpha})=(\varphi_{\sigma_{\alpha},\chi_{\alpha}} ,t_{\alpha
})$.} By Lemma~\ref{close} $\mathfrak{M}(A_{0}+AP_{+})\backslash
\underline{\mathbb{C}_{\scriptscriptstyle \geq0}}$ is closed, which means that $\varphi
\in\mathfrak{M}(A_{0}+AP_{+})\backslash\underline{\mathbb{C}_{\scriptscriptstyle \geq0}}$.
Therefore%
\begin{equation}
H(\varphi_{\alpha},t_{\alpha})(f_{\scriptscriptstyle A_{0}})=0=H(\varphi,t)(f_{\scriptscriptstyle A_{0}})\text{
\ \ for all }f_{\scriptscriptstyle A_{0}}\in A_{0}.\label{stars}%
\end{equation}
From here on, we therefore only consider the case when $f\in AP_{+}.$ Firstly,
if $t_{\alpha}\equiv1,$ then $t=1$ and%
\[
H(\varphi_{\alpha},t_{\alpha})=H(\varphi_{\alpha},1)=\underline{+\infty
}=H(\varphi,1)=H(\varphi,t).
\]
If each $t_{\alpha}\in\lbrack0,1)$, then%
\[
\varphi_{\alpha}(f)=\sum_{k\geq0}f_{k}e^{-\sigma_{\alpha}t_{k}}\chi_{\alpha
}(t_{k}),
\]
for $f\in AP_{+}.$ Moreover, for $T>0,$%
\[
H(\varphi_{\alpha},t_{\alpha})(e^{-sT})=e^{-(\sigma_{\alpha}-\log(1-t_{\alpha
}))T}\chi_{\alpha}(T).
\]
We will now consider two separate cases.
\begin{itemize}
\item[(2a)] $\varphi\neq\underline{+\infty}.$ Then $\varphi_{\alpha
}\rightarrow\varphi\neq \underline{+\infty}$ so for $T>0,$%
\[
\varphi_{\alpha}(e^{-sT})=e^{-\sigma_{\alpha}T}\chi_{\alpha}(T)\rightarrow
e^{-\sigma T}\chi(T)=\varphi(e^{-sT}).
\]
If $t<1,$ then this implies that%
\[
H(\varphi_{\alpha},t_{\alpha})(e^{-sT})=e^{-(\sigma_{\alpha}-\log(1-t_{\alpha
}))T}\chi_{\alpha}(T)\rightarrow e^{-(\sigma-\log(1-t))T}\chi(T)=H(\varphi,t),
\]
and we are done. If $t=1$ on the other hand, then%
\[
e^{-(\sigma_{\alpha}-\log(1-t_{\alpha}))T}\chi_{\alpha}(T)\rightarrow
0=\underline{+\infty}(e^{-sT}),
\]
and so $H(\varphi_{\alpha},t_{\alpha})(e^{-sT})\rightarrow H(\varphi
,t)(e^{-sT}).$
\item[(2b)] $\varphi=\underline{+\infty}.$ Since $\varphi_{\alpha}%
\rightarrow\underline{+\infty},$ we have that $\varphi_{\alpha}(e^{-sT}%
)\rightarrow \underline{+\infty}(e^{-sT})=0$ for $T>0.$ That is, $e^{-\sigma
_{\alpha}T}\chi_{\alpha}(T)\rightarrow0.$ Therefore,%
\[
e^{-(\sigma_{\alpha}-\log(1-t_{\alpha}))T}\chi_{\alpha}(T)\rightarrow0,
\]
and $H(\varphi_{\alpha},t_{\alpha})(e^{-sT})\rightarrow H(\varphi
,t)(e^{-sT}).$
\end{itemize}
Due to (\ref{stars}) this completes the proof in the case when $(\varphi
_{\alpha},t_{\alpha})=(\varphi_{\sigma_{\alpha},\chi_{\alpha}},t_{\alpha})$.

We have now shown that $H$ converges for all three types of subnets
$(\varphi_{\alpha},t_{\alpha})$ which shows that $\mathfrak{M}(A_{0}+AP_{+})$
is contractible.
\end{proof}

\begin{corollary}
$A_{0}+AP_{+}$ is Hermite
\end{corollary}

As mentioned before, this corollary is a direct consequence of \cite[Theorem
3, p. 127]{VYL}. Yet another result, which follows immediately from  \cite[Corollary
1.4]{BS}, concerns projective freeness.

\begin{corollary}
$A_{0}+AP_{+}$ is a projective free ring.
\end{corollary}

\section{$A_{0}+AP_{+}$ is not a GCD domain}

The relevant definitions were presented in Section \ref{prel}, and now we will
prove that $A_{0}+AP_{+}$ is not a GCD domain, by giving an example of two
elements in $A_{0}+AP_{+}$ which do not have a gcd. The method of proof is the same as in \cite{MR0}.
The two elements we will
consider are
\begin{equation}
F_{1}=\frac{1}{(1+s)}\ \ \ \text{and \ \ }F_{2}=\frac{1}{(1+s)}e^{-1/s}%
.\label{FF}%
\end{equation}
Note that both $F_{1}$ and $F_{2}$ belong to $A_{0}$, and thereby to
$A_{0}+AP_{+}$. We shall now state some preliminary results used in the proof.
We begin with Cohen's factorization theorem; see for example \cite[Theorem 1.6.5]{Br}.

\begin{proposition}
\label{PropV} Let $R$ be a Banach algebra
with a bounded left approximate identity. Then for every sequence
$(a_{n})_{n\geq1}$ in $R$ converging to zero, there exists a sequence
$(b_{n})_{n\geq1}$ in $R$ converging to zero, as well as an element $c\in R$
such that $a_{n}=cb_{n}$ for all $n\geq1.$
\end{proposition}

\begin{lemma}
\label{LAI}The maximal ideal $\mathfrak{M}_{0}:=\{f\in A_{0}+AP_{+}%
\ |\ f(0)=0\}$ has a bounded left approximate identity.
\end{lemma}

\begin{proof}
Let%
\[
e_{n}=\frac{s}{s+\frac{1}{n}}.
\]
We shall now prove that this is a bounded approximate identity for
$A_{0}+AP_{+}.$ First of all, we note that%
\[
e_{n}=1-\underset{=: e_{\scriptscriptstyle A_{0}}^{n}\in A_{0}}{\underbrace{\frac{1}{n}\left(  \frac{1}%
{s+\frac{1}{n}}\right)  }}=:1+e_{\scriptscriptstyle A_{0}}^{n}\in A_{0}+AP_{+}.
\]
Furthermore, $e_{n}(0)=0,$ and so in fact, $e_{n}\in\mathfrak{M}_{0}$. Moreover,%
\[
\|e_{n}\|_{\mathcal{\infty}}=\left\Vert 1-\frac{1}{n}\left(  \frac{1}%
{s+\frac{1}{n}}\right)  \right\Vert _{\infty}=1,
\]
so $e_{n}$ is bounded. We need to prove that%
\begin{equation}
\lim_{n\rightarrow\infty}\left\Vert e_{n}f-f\right\Vert _{\mathcal{\infty}%
}=0.\label{limbai}%
\end{equation}
Now, fix $n$ and let $f\in\mathfrak{M}_{0}$, then%
\[
e_{n}f-f=(1+e_{\scriptscriptstyle A_{0}}^{n})f-f=e_{\scriptscriptstyle A_{0}}^{n}f\in A_{0}.
\]
Therefore, we only need to prove that
\[
\lim_{n\rightarrow\infty}\left\Vert e_{\scriptscriptstyle A_{0}}^{n}f\right\Vert _{\infty}=0.
\]
Let $\varepsilon>0$ be given. We know that
$$
\lim_{|s|\rightarrow0}f(s)=0
$$
since $f\in\mathfrak{M}_{0},$ so there exists $\delta>0$ such that
$|f(s)|<\varepsilon$ for all $s$ with $|s|<\delta$. For $|s|<\delta,$%
\[
|e_{\scriptscriptstyle A_{0}}^{n}f|=\left\vert \frac{1}{n}\left(  \frac{1}{s+\frac{1}{n}}\right)
f(s)\right\vert <\varepsilon,
\]
and this is independent of our choice of $n.$ When $|s|\geq\delta$, we
have%
\[
|e_{\scriptscriptstyle A_{0}}^{n}f|=\left\vert \frac{1}{n}\left(  \frac{1}{s+\frac{1}{n}}\right)
f(s)\right\vert \leq\left\vert \frac{1}{\delta n+1}f(s)\right\vert .
\]
However, since $f\in A_{0}+AP_{+}$, $|f|$ is bounded. Hence, for
\[
n>\frac{\frac{\|f\|_{\infty}}{\varepsilon}-1}{\delta},
\]
$|e_{\scriptscriptstyle A_{0}}^{n}(s)f(s)|<\varepsilon$ for all $|s|\geq\delta,$ and hence, for
all $s.$ Therefore (\ref{limbai}) holds, and $(e_{n})$ is a bounded approximate
identity for $\mathfrak{M}_{0}$.
\end{proof}

\begin{theorem}
$A_{0}+AP_{+}$ is not a GCD domain.
\end{theorem}

\begin{proof}
We claim that $F_{1},F_{2}$ in (\ref{FF}) have no gcd. Suppose, on the
contrary that $D$ is a gcd for $F_{1}$ and $F_{2}$, so that%
\[
F_{1}=DQ_{1}\ \ \ \text{and \ \ }F_{2}=DQ_{2},
\]
for some $Q_{1},Q_{2}\in A_{0}+AP_{+}.$

Suppose that $Q_{1}(0)\neq0.$ Then, at least in a neighbourhood of zero,
\[
e^{-1/s}=\frac{F_{2}(s)}{F_{1}(s)}=\frac{Q_{2}(s)}{Q_{1}(s)}.
\]
However, this is impossible because%
\[
\lim_{\mathbb{R\backepsilon\ \omega\rightarrow}0}e^{-1/(i\omega)}\text{ does
not exist,}%
\]
while
\[
\lim_{\mathbb{R\backepsilon\ \omega\rightarrow}0}\frac{Q_{2}(i\omega)}%
{Q_{1}(i\omega)}=\frac{Q_{2}(0)}{Q_{1}(0)},
\]
since $Q_{1}(0)\neq0$ and both $Q_{1}$ and $Q_{2}$ are continuous. Hence,
$Q_{1}(0)=0.$

Similarly, we suppose that $Q_{2}(0)\neq0.$ Then, in a neighbourhood of zero,
\[
e^{1/s}=\frac{F_{1}(s)}{F_{2}(s)}=\frac{Q_{1}(s)}{Q_{2}(s)}.
\]
However, this is impossible because%
\[
\lim_{\mathbb{R\backepsilon\ \omega\rightarrow}0}e^{1/(i\omega)}\text{ does
not exist,}%
\]
while
\[
\lim_{\mathbb{R\backepsilon\ \omega\rightarrow}0}\frac{Q_{2}(i\omega)}%
{Q_{1}(i\omega)}=\frac{Q_{2}(0)}{Q_{1}(0)}%
\]
does. We have thus concluded that $Q_{1}(0)=Q_{2}(0)=0,$ and thus both $Q_{1}
$ and $Q_{2}$ belong to the maximal ideal $\mathfrak{M}_{0}:=\{f\in
A_{0}+AP_{+}\ |\ f(0)=0\}.$ By Lemma \ref{LAI}, $\mathfrak{M}_{0}$ has a
bounded approximate identity, and then, by Proposition \ref{PropV} there
exists a $G\in\mathfrak{M}_{0}$ such that $G$ is a common factor for $Q_{1}$
and $Q_{2}$. Hence, $K:=DG$ is also a divisor of both $F_{1}$ and $F_{2}$.
Since $D$ is a gcd for $F_{1}$ and $F_{2}$, $K$ must divide $D,$ and so%
\begin{equation}
\underset{=K}{\underbrace{DG}}H=D,\label{gh}%
\end{equation}
for some $H\in A_{0}+AP_{+}.$ By definition $F_{1}\neq0$ for $s\neq\infty$, so
the same must hold for $D$. Therefore, by (\ref{gh}), $GH=1$ for
$s\in\mathbb{C}_{\scriptscriptstyle \geq0}\backslash\{+\infty \}$. But, $G\in\mathfrak{M}_{0}$ and
$H\in A_{0}+AP_{+}$ is bounded, so%
\[
\lim_{s\rightarrow0}G(s)H(s)=0,
\]
which is a contradiction. Thus, $F_{1}$ and $F_{2}$ have no gcd in
$A_{0}+AP_{+}$, and $A_{0}+AP_{+}$ is not a GCD domain.\
\end{proof}

\section{$A_{0}+AP_{+}$ is not a pre-Bezout domain}

The proof that $A_{0}+AP_{+}$ is not a pre-Bezout domain relies on the notion
of an approximate identity, Proposition~\ref{PropV} and the corona theorem.
The method of proof is the same as in \cite{MR0}.

\begin{theorem}
$A_{0}+AP_{+}$ is not a pre-Bezout domain.
\end{theorem}
\begin{proof}Consider the following two elements
in $A_0+AP_+$:
$$
U_1=\frac{1}{s+1}\textrm{ and }U_2=e^{-s}.
$$
As $U_1$ is outer and $U_2$ is inner in the Hardy algebra of the right half-plane,
it can be seen from the inner-outer factorization of $H^\infty$ functions that
the pair $(U_1,U_2)$ has $1$ as a greatest common divisor in $A_0+AP_+$.
Suppose that $A_0+AP_+$ is a pre-Bezout domain. Then there exist
$X$, $Y$ in $A_0+AP_+$ such that
$1 = U_1 \cdot  X + U_2 \cdot  Y$. Passing  the
limit as $s\rightarrow  + \infty$, we arrive at the contradiction
that $1 = 0$. Hence $A_{0}+AP_{+}$ is not a pre-Bezout domain.
\end{proof}

\section{$A_{0}+AP_{+}$ is not coherent}

The proof we give in this section is based on the same method used to
show the noncoherence of the causal Wiener algebra $W^+$ in \cite{MR}.
We will begin with two lemmas, the first lemma is a special case of Nakayama's
lemma, see for instance \cite[Theorem 2.2]{Ma}, while the second one can be
proved using the same arguments as \cite[Lemma 2.8]{Sas08}.

\begin{lemma}
\label{L1}Let $L\neq0$ be an ideal in $A_{0}+AP_{+},$ contained in the maximal
ideal $\mathfrak{M}_{0}$. If $L=L\mathfrak{M}_{0},$ then $L$ cannot
be finitely generated.
\end{lemma}

\begin{lemma}
\label{L2}Let $f_{1},f_{2}\in\mathfrak{M}_{0},$ and let $\delta>0.$ Let
$\textrm{\em Inv}(A_{0}+AP_{+})$ denote the set of all invertible elements in $A_{0}%
+AP_{+}.$ Then there exists a sequence $(g_{n})_{n\geq0}$ in $A_{0}+AP_{+}$
such that
\begin{enumerate}
\item $g_{n}\in \textrm{\em Inv}(A_{0}+AP_{+})$ for all $n\geq0,$

\item $g_{n}$ converges to a limit $g\in\mathfrak{M}_{0}$ in $A_{0}+AP_{+},$

\item $||g_{n}^{-1}f_{i}-g_{n+1}^{-1}f_{i}||_{\mathcal{A}}\leq\delta/2^{n}$,
for all $n\geq0$, $i=1,2.$
\end{enumerate}
\end{lemma}

\begin{theorem}
$A_{0}+AP_{+}$ is not coherent.
\end{theorem}

\begin{proof}
We will proceed much in line with \cite[Theorem 1.3]{Sas08}. Let%
\[
p=(1-e^{-s})^{3},\ \ \ S=e^{-\frac{1+e^{-s}}{1-e^{-s}}}.
\]
We note that
\[
(1-z)^{3}e^{-\frac{1+z}{1-z}}=\sum_{k=0}^{\infty}a_{k}z^{k},\text{ \ \ for
}z\in\overline{\mathbb{D}},\ \sum_{k=0}^{\infty}|a_{k}|<\infty,
\]
by \cite[Remark, p. 224]{MR}. Since $e^{-s}\in\overline{\mathbb{D}}$ for
$s\in\mathbb{C}_{\geq0},$ $pS\in A_{0}+AP_{+},$ and by definition
$p\in\mathfrak{M}_{0}$. Let $I=(p)$ and $J=(pS)$ be the ideals, (finitely)
generated by $p$ respectively $pS$. To prove that $A_{0}+AP_{+}$ is not
coherent, we shall show that $I\cap J$ is not finitely generated.

The first step is to characterize $I\cap J.$ Let,%
\[
K=\{pSf:f,Sf\in A_{0}+AP_{+}\}.
\]
We claim that $K=I\cap J.$ By definition, $K\subset(I\cap J),$ so we only need
to show the reverse inclusion. Let $g\in I\cap J$, then there exist two
functions $f,h\in A_{0}+AP_{+}$ such that%
\[
g=ph=pSf.
\]
Since $p\neq0$ and since $A_{0}+AP_{+}$ is an integral domain, $h=Sf\in
A_{0}+AP_{+}$ and so $g\in K$ and $K=I\cap J$.

Now we shall prove that $K=I\cap J$ is not finitely generated. To do this, we
define%
\[
L=\{f\in A_{0}+AP_{+}:Sf\in A_{0}+AP_{+}\}.
\]
Then $K=pSL,$ and since $S$ has a singularity at $s=0$ we have that
$L\subset\mathfrak{M}_{0}.$ We shall now prove that $\ L=L\mathfrak{M}_{0}$.
Then by Lemma \ref{L1} $L$ cannot be finitely generated, and neither can $K$,
so $A_{0}+AP_{+}$ is not coherent. To prove that $L=L\mathfrak{M}_{0},$ let
$f\in L$ be arbitrary. We want to find factors $h\in L,g\in$ $\mathfrak{M}%
_{0}$ such that $f=hg$. Let $f_{1}=f\in\mathfrak{M}_{0}$, $f_{2}%
=Sf\in\mathfrak{M}_{0}.$ Then, by Lemma \ref{L2}, for every $\delta>0$ there
exists a sequence $(g_{n})_{n\geq1}$ in $A_{0}+AP_{+}$ such that properties
(1)-(3) of Lemma \ref{L2} hold. Let%
\[
h_{n}=g_{n}^{-1}f,\ \ \ H_{n}=g_{n}^{-1}Sf.
\]
Then $h_{n},H_{n}\in\mathfrak{M}_{0}$ and (3) in Lemma \ref{L2} imply that
$(h_{n})_{n\geq1},(H_{n})_{n\geq1}$ are Cauchy sequences in $A_{0}+AP_{+}.$
Since $\mathfrak{M}_{0}$ is closed, these sequences converge to $h$
respectively $H$, $h,H\in\mathfrak{M}_{0}.$ Since the limit of a sequence of
holomorphic functions is unique,
$$
\lim_{n\rightarrow\infty}Sh_{n}=Sh=H,
$$
and
since both $h$ and $Sh$ are in $\mathfrak{M}_{0}$, $h\in L$. Further, by (2) in
Lemma \ref{L2},
\[
f=\lim_{n\rightarrow\infty}h_{n}g_{n}=hg,
\]
where $h\in L$ and $g\in\mathfrak{M}_{0}.$ That is, $L=L\mathfrak{M}_{0}$,
which completes the proof.
\end{proof}

\medskip 

\noindent {\bf Acknowledgements:} The authors thank the anonymous referee 
for the careful review and the several suggestions which improved the presentation 
of the article.

\end{document}